\documentclass[12pt]{amsart}
\usepackage[utf8]{inputenc}
\usepackage{amsmath,amssymb,amsthm}
\usepackage[all]{xy}

\newcommand{\Z}{{\mathbb Z}}
\newcommand{\Q}{{\mathbb Q}}
\newcommand{\PP}{{\mathbb P}}
\newcommand{\Tr}{\operatorname{Tr}}

\newcommand{\GL}{\operatorname{GL}}

\newcommand{\Spec}{\operatorname{Spec}}

\newcommand{\isom}{\cong}
\newcommand{\ra}{\longrightarrow}

\newcommand{\bu}{\bullet}
\newcommand{\leqs}{\leqslant}
\newcommand{\geqs}{\geqslant}

\providecommand{\mc}[1]{\mathcal{ #1} }

\usepackage[margin=1in]{geometry}

\newtheorem{theorem}{Theorem}[section]
\newtheorem{proposition}[theorem]{Proposition}
\newtheorem{lemma}[theorem]{Lemma}

\theoremstyle{definition}

\newtheorem{definition}[theorem]{Definition}
\newtheorem{remark}[theorem]{Remark}
\newtheorem{example}[theorem]{Example}

\numberwithin{table}{section}

\begin{document}
\date{16th September 2021}
\title[Computing structure constants]{Computing structure constants
for rings \\ of finite rank from minimal free resolutions}

\author{Tom Fisher}
\address{University of Cambridge,
          DPMMS, Centre for Mathematical Sciences,
          Wilberforce Road, Cambridge CB3 0WB, UK}
\email{T.A.Fisher@dpmms.cam.ac.uk}

\author{Lazar~Radi\v{c}evi\'{c}}
\address{University of Cambridge,
          DPMMS, Centre for Mathematical Sciences,
          Wilberforce Road, Cambridge CB3 0WB, UK}
\email{lazaradicevic@gmail.com}

\begin{abstract}
  We show how the minimal free resolution of a set of $n$ points in
  general position in projective space of dimension $n-2$ explicitly
  determines structure constants for a ring of rank $n$. This
  generalises previously known constructions of Levi-Delone-Faddeev
  and Bhargava in the cases $n=3,4,5$.
\end{abstract}

\maketitle

\section{Introduction}

A classical construction, known as the Levi-Delone-Faddeev
correspondence \cite{levi,delone1964theory} (see also \cite{ggs}),
shows that a binary cubic form
\[ f(x,y) = a x^3 + b x^2 y + c x y^2 + d y^3 \] naturally determines
a ring of rank $3$.  Explicitly, if $\xi$ is a symbol formally
satisfying $f(\xi,1)=0$ then $\omega = a \xi$ and
$\theta = -d \xi^{-1}$ satisfy the relations
\begin{equation}
\label{cubic_ring}
\begin{aligned}
\omega^2 &= -ac - b \omega + a \theta, \\
\omega \theta &= -ad, \\
\theta^2 &= -bd - d \omega + c \theta.
\end{aligned}
\end{equation}
These relations may be used to define a commutative and associative
multiplication on the free module with basis $1,\omega,\theta$. The
construction works over any base ring, and gives a discriminant
preserving bijection between equivalence classes of binary cubic forms
and isomorphism classes of rings of rank $3$. This construction was
extended to rings of rank $4$ and $5$ by Bhargava
\cite{bhargava2004higher, bhargava2008higher}, who considered
pairs of quadratic forms in $3$ variables, and $5 \times 5$
alternating matrices of linear forms in $4$ variables.

We describe an extension to rings of rank $n$ for any integer
$n \geqs 3$. Our main result (Theorem~\ref{algebra constructed from n
  points}) shows how a set of $n$ points $X \subset \PP^{n-2}$ in
general position determines, by means of an explicit construction
involving the minimal free resolution of $X$, structure constants for
an algebra $A$ of rank $n$. In particular we show that this algebra
$A$ is isomorphic to the coordinate ring of $X$.

We should say straight away that we are not expecting to fully
generalise Bhargava's work, and count number fields of degree $n >
5$. Instead our motivation comes from the study of genus one curves.
If $C \subset \PP^{n-1}$ is a genus one curve of degree $n$, embedded
by a complete linear system, then a generic hyperplane section of $C$
will be a set of $n$ points in general position.  The first author
showed in \cite{fisher2018invforalln} how to associate to such a curve
$C$ a matrix of quadratic forms $\Omega$ describing the invariant
differential.  One application of Theorem~\ref{algebra constructed
  from n points} is that the associative law then determines some of
the equations defining the space of all such $\Omega$'s. Another
application is given by the second author in his PhD thesis
\cite{LazarThesis} where for $E/\Q$ an elliptic curve and $n \geqs 2$
an integer, he gives a simple bound on the least discriminant of a
degree $n$ number field over which each element of order $n$ in the
Tate Shafarevich group of $E$ capitulates.

Having mentioned these applications to the study of curves, in the
rest of this article we only consider finite sets of points in
projective space.
%rather than curves.

For the statement of Theorem~\ref{algebra constructed from n points}
we work over a field $K$ of characteristic $0$. However, examination
of the proofs shows that all we need is that the characteristic does
not divide $2n$. In Section~\ref{sec:orders} (with the main proof
postponed to Section~\ref{sec:newproof}) we describe a slightly more
complicated variant of our construction that works in all
characteristics. It is this construction that reduces (in the cases
$n=3,4,5$) to the earlier work of Levi-Delone-Faddeev and Bhargava. It
also gives better bounds in \cite[Theorem 1.0.1]{LazarThesis}.

In Section~\ref{sec:bijection} we review the connection between
non-degenerate algebras of dimension $n$ and sets of $n$ points in
$\PP^{n-2}$ in general position.  Then, as explained in
Section~\ref{sec:overview}, the proof of Theorem~\ref{algebra
  constructed from n points} comes down to (i) checking our
construction of the structure constants behaves well under all changes
of co-ordinates, and (ii) checking that the theorem holds for the
standard set of $n$ points:
\[(1:0:\ldots:0), \,\, (0:1:0:\ldots:0), \,\, \ldots \,\, , (0:
  \ldots :0:1), \,\, (1:1: \ldots :1).\] We give the proof of (i)
in Sections~\ref{sec:sym} and~\ref{sec:change}. We may check (ii) for
any given $n$ by computer algebra. We give a proof that works for all
$n$ in Sections~\ref{structure constants section} and~\ref{wilson
  description section}, using an explicit description of the minimal
free resolution due to Wilson~\cite{wilson}.

\subsection*{Acknowledgements}
This article is based on part of the second author's PhD thesis.  We
thank Manjul Bhargava and Melanie Wood for useful conversations,
% for pointing us to the work of Wilson~\cite{wilson},
and Jack Thorne for alerting us to an oversight in
an earlier version of Section~\ref{sec:bijection}.

\section{Statement of the main theorem}
\label{sec:state}

We recall a few basic notions from commutative algebra that will be
needed to state our main theorem. Throughout, we work over a field $K$
with algebraic closure $\bar{K}$.  Let $R=K[x_1,\ldots,x_m]$ be the
polynomial ring with its usual grading. For $M=\oplus_d M_d$ a graded
$R$-module, we write $M(c)=\oplus_d M_{c+d}$ for the graded $R$-module
with grading shifted by $c$. A direct sum of modules of the form
$R(c)$ is a called a {\em graded free $R$-module}.
 
\begin{definition}\label{def minimal free res}
  A {\em graded free resolution} of a graded $R$-module $M$ is a chain
  complex $F_{\bullet}$ of graded free $R$-modules
  \[
    F_r \xrightarrow{\,\,\,\phi_{r}\,\,\,} F_{r-1}
    \xrightarrow{\phi_{r-1}} \ldots \xrightarrow{\,\,\, \phi_2 \,\,\,}
    F_1 \xrightarrow{\,\,\,\phi_1\,\,\,} F_0,
  \]
  that is exact in degree $>0$, and has
  $H_0(F_{\bullet})=F_0/\phi_1(F_1) \cong M$.  Let
  $\mathfrak{m}=(x_1,\ldots,x_{m})$ be the maximal homogeneous ideal
  of $R$. We say a resolution $F_{\bullet}$ is {\em minimal} if we
  have $\phi_{k}(F_k) \subset \mathfrak{m} F_{k-1}$ for every
  $k \geqs 1$.
\end{definition}

Our interest is in the case $M\cong R/I$, where $I$ is a homogeneous
ideal in $R$.

\begin{remark}
  The minimal free resolution of a module is unique up to an
  isomorphism of chain complexes. Any such isomorphism consists of
  changes of bases for the free $R$-modules $F_k$ in the resolution,
  see \cite[Theorem~20.2]{eisenbud}.
\end{remark}

We require that the maps $\phi_k$ respect the grading of the
modules. For example, a homomorphism of modules $R(-m) \to R$ is
defined by multiplication by an element $f \in R$, and this is a
graded homomorphism if and only if $f$ is homogeneous of degree $m$
(or zero).  By choosing bases for each module $F_k$ in the resolution,
we may represent the maps $\phi_k$ as matrices of homogeneous
polynomials. By abuse of notation we also write $\phi_k$ for these
matrices.

\begin{remark}
  The condition that the resolution is minimal means that every
  non-zero entry of every matrix has positive degree. By Nakayama's
  lemma, this is equivalent to requiring that $\phi_k$ takes the basis
  of $F_k$ to a minimal set of generators for the kernel of
  $\phi_{k-1}$, see \cite[Corollary~1.5]{eisenbud2005geometry}. This
  characterisation makes it clear that every finitely generated graded
  module admits a minimal free resolution.
\end{remark}

A minimal free resolution of an ideal contains the data of a set of
generators for the ideal, the data of all relations (syzygies) that
these generators satisfy, the data of relations that these relations
satisfy, and so on iteratively. We illustrate this in the following
example.
    
\begin{example}
\label{ex:res:n=4}
Let $X$ be the set of four points $(1:0:0)$, $(0:1:0)$, $(0:0:1)$,
$(1:1:1)$ in $\PP^{2}$. The homogeneous ideal $I$ of $X$ in
$R=K[x_1,x_2,x_3]$ is generated by the quadratic forms
$A=x_1(x_2-x_3)$ and $B=x_2(x_1-x_3)$.  For the first step of the
resolution, we can take $F_0=R$, $F_1=R(-2)^2$, and let
$\phi_1 : R(-2)^2 \xrightarrow{} R$ be the map represented by the row
matrix $(A,B)$, so that $\mathrm{coker}(\phi_1) \cong R/I$.  To
compute the second step, we observe that $A$ and $B$ satisfy the
relation $B \cdot A + (-A) \cdot B=0$. Furthermore, any equation
$f \cdot A+g \cdot B=0$ is obtained by multiplying this relation by
some $r \in R$, i.e., we have $f=r \cdot B$ and $g=-r \cdot A$. We now
take $F_2=R(-4)$, and let $\phi_2 :R(-4) \xrightarrow{} R(-2)^2$ be
the map represented by the column matrix $(B,-A)^T$. Since this map is
injective, this is where the resolution stops.  We obtain the chain
complex
\begin{equation}
\label{res:n=4}
0 \ra R(-4) \xrightarrow{(B,-A)^T} R(-2)^2 \xrightarrow{(A,B)} R \ra 0
\end{equation}
which is exact at the middle term and on the left, and is hence the
minimal free resolution of $R/I$.
\end{example}

More generally we consider sets of points of the following form.

\begin{definition}
  A zero dimensional variety $X \subset \PP^{n-2}$ defined over $K$ is
  a {\em set of $n$ points in general position}, if $X$ has degree $n$
  and the set of geometric points $X(\bar{K})$ consists of $n$ points
  in general position, meaning that no subset of $X(\bar{K})$ of size
  $n-1$ is contained in a hyperplane.
\end{definition}

\begin{theorem}
\label{min res theorem}
Let $n \geqs 4$, and let $R=K[x_1,\ldots,x_{n-1}]$ be the coordinate
ring of $\PP^{n-2}$.  Let $X \subset \PP^{n-2}$ be a set of $n$ points
in general position. Let $I \subset R$ be the homogeneous ideal of
$X$.  Then $I$ is (arithmetically) Gorenstein, and the minimal free
resolution $F_{\bullet}$ of $R/I$ takes the form 
\begin{equation*}
% \label{free res}
\begin{split}
  0 \ra R(-n) \xrightarrow{\phi_{n-2}} R(-n+2)^{b_{n-3}} &
  \xrightarrow{\phi_{n-3}} R(-n+3)^{b_{n-4}} \xrightarrow{\phi_{n-4}}
  \ldots\\& \ldots \xrightarrow{\,\,\,\phi_{3}\,\,\,} R(-3)^{b_{2}}
  \xrightarrow{\,\,\,\phi_{2}\,\,\,} R(-2)^{b_{1}}
  \xrightarrow{\,\,\,\phi_{1}\,\,\,} R \ra 0,
\end{split}
\end{equation*}
where the Betti numbers are given by $b_i=n\binom{n-2}{i} - \binom{n}{i+1}$.
\end{theorem}

\begin{proof}
  The minimal free resolution is as described in
  \cite[Theorem~138]{wilson}, and the references cited there.  For the
  statement that $I$ is Gorenstein see \cite[Corollary~140]{wilson}.
\end{proof}	
 
We note that $\phi_{1}$ and $\phi_{n-2}$ are represented by matrices
of quadratic forms, while the maps $\phi_i$, for $1<i<n-2$, are
represented by matrices of linear forms.
 
\begin{definition}
\label{def:brackets}
The resolution $F_{\bullet}$ determines the following quadratic forms
in $x_1, \ldots,x_{n-1}$.
\begin{enumerate}
\item For $1\leqs a_1,a_2,\ldots,a_{n-2} \leqs n-1$ we define
  \begin{equation*}
    [a_1,a_2,\ldots,a_{n-2}]_{F_{\bu}}=\frac{\partial \phi_{1}}{\partial x_{a_1}} \frac{\partial \phi_{2}}{\partial x_{a_2}} \cdot \cdot \cdot \frac{\partial \phi_{n-2}}{\partial x_{a_{n-2}}},
  \end{equation*}
  where the partial derivative of a matrix is the matrix of partial
  derivatives of its entries, and the product is matrix
  multiplication.
\item Let $\sigma$ be the $(n-2)$-cycle $(12\ldots n-2)$ in the
  symmetric group $S_{n-2}$. We define
  \[ [[a_1,a_2,\ldots,a_{n-2}]]_{F_{\bu} }=\sum_{k=1}^{n-2}
    [a_{\sigma^{2k}(1)},a_{\sigma^{2k}(2)},
    \ldots,a_{\sigma^{2k}(n-2)}]_{F_{\bu}}. \]
\item For $1 \leqs j \leqs n-1$ we define
  $\Omega_j = (-1)^j [[ 1,2, \ldots, \widehat{j}, \ldots,n-1
  ]]_{F_{\bu}}$.
\end{enumerate}
The choice of resolution will usually be fixed, and we therefore drop
the subscripts $F_{\bu}$.
\end{definition}

For our main result we work over a field $K$ of characteristic $0$.

\begin{theorem} \label{algebra constructed from n points} Let
  $X \subset \PP^{n-2}$ be a set of $n$ points in general position,
  and let $\Omega_1,\ldots,\Omega_{n-1}$ be the quadratic forms
  associated to a minimal free resolution of $X$. Then there exists a
  commutative and associative $K$-algebra $A$, of dimension $n$, and a
  $K$-basis $1=\alpha_0,\alpha_1,\ldots,\alpha_{n-1}$ for $A$, such
  that for each $1\leqs i,j \leqs n-1$ we have
  \[
    \alpha_i \alpha_j =c^{0}_{ij}+\sum_{k=1}^{n-1} \frac{\partial^2
      \Omega_k}{\partial x_i \partial x_j}\alpha_k,
  \]
  for some constant $c^{0}_{ij} \in K$. Moreover $A$ is isomorphic to
  the affine coordinate ring (i.e., ring of global functions) of $X$,
  and the $\alpha_i$ for $1 \leqs i \leqs n-1$ span the trace zero subspace.
\end{theorem}

\begin{remark}
\label{rem:c0}
Following \cite[page~68]{bhargava2008higher} we can use the
associative law to solve for the $c^{0}_{ij}$. Explicitly, for any
$1\leqs i,j,k \leqs n-1$ with $i \not= k$, comparing coefficients of
$\alpha_k$ in
$\alpha_i (\alpha_j \alpha_k) = (\alpha_i \alpha_j)\alpha_k$ gives
\[ c^{0}_{ij}=\sum_{r=1}^{n-1} \left( \frac{\partial^2
      \Omega_r}{\partial x_j \partial x_k} \frac{\partial^2
      \Omega_k}{\partial x_r \partial x_i}- \frac{\partial^2
      \Omega_r}{\partial x_i \partial x_j} \frac{\partial^2
      \Omega_k}{\partial x_r \partial x_k}\right). \]
\end{remark}

\begin{remark}\label{defined up to scalar}
  The construction of $[a_1,a_2,\ldots,a_{n-2}]$, and hence of
  $\Omega_1, \ldots, \Omega_{n-1}$, is independent of the choice of
  basis for the free $R$-modules in the resolution $F_\bullet$, except
  for the leftmost module $R(-n)$.  The quadratic forms
  $\Omega_1, \ldots, \Omega_{n-1}$ are therefore uniquely determined
  up to multiplying through by an overall scalar. It is clear that
  this gives an isomorphic $K$-algebra.
\end{remark}

\section{Constructing orders in number fields}
\label{sec:orders}

In this section we explain the connection between Theorem~\ref{algebra
  constructed from n points} and the previously known constructions
due to Levi-Delone-Faddeev and Bhargava for $n=3,4,5$. Whereas we work with
algebras over a field of characteristic zero, the latter constructions
work for rings of rank~$n$, i.e., algebras over $\Z$. We discuss to
what extent this earlier work generalises to larger~$n$.

Let $A$ be an $n$-dimensional commutative $K$-algebra with $K$-basis
$1, \alpha_1, \ldots, \alpha_{n-1}$.  The structure constants
$c_{ij}^{k}$ for $1 \leqs i,j,k \leqs n-1$ are determined by
\[ \alpha_i \alpha_j = c_{ij}^{0} + \sum_{k=1}^{n-1} c_{ij}^{k}
  \alpha_{k}. \] As noted in Remark~\ref{rem:c0}, the $c_{ij}^{0}$ may
be recovered from the other structure constants using the associative
law.  We say that bases $1, \alpha_1, \ldots, \alpha_{n-1}$ and
$1, \beta_1, \ldots, \beta_{n-1}$ differ by a {\em shear} if
$\beta_i = \alpha_i + \lambda_i \cdot 1$ for some
$\lambda_1, \ldots, \lambda_{n-1} \in K$.

When $n = 3$ the algebra constructed by Levi-Delone-Faddeev
% \cite{delone1964theory}
(as defined by~\eqref{cubic_ring}
in the introduction) is uniquely determined, up to shear, by
\begin{equation}
    \label{strconsts:3}
    c_{11}^{2} = a, \qquad 
    c_{11}^{1} - 2 c_{12}^{2} = -b, \qquad 
    c_{22}^{2} - 2 c_{12}^{1} = c, \qquad
    c_{22}^{1} = -d. 
\end{equation}
To compare with the algebra in Theorem~\ref{algebra constructed from n
  points} we consider the minimal free resolution
\[ 0 \ra R(-3) \stackrel{f}{\ra} R \ra 0 \] where
$f(x_1,x_2) = a x_1^3 + b x_1^2 x_2 + c x_1 x_2^2 + d x_2^3$.
Using Definition~\ref{def:brackets}, we compute
\[ \Omega_1 = -[[2]] = -[2] = -\frac{\partial f}{\partial x_2} = -b
  x_1^2 - 2c x_1 x_2 - 3 d x_2^2 \] and
\[ \Omega_2 = [[1]] = [1] = \frac{\partial f}{\partial x_1} = 3a x_1^2
  + 2b x_1 x_2 + c x_2^2. \] From this it is easy to check that
\[c_{ij}^{k} = \frac{1}{6} \frac{ \partial^2 \Omega_k}{\partial x_i
    \partial x_j}\] is a solution to~\eqref{strconsts:3}.

When $n \geqs 4$ the algebra is uniquely determined, up to shear,
by the linear combinations of structure constants in the left hand
column of Table~\ref{table:strconst}, where $i,j,k$ range over all
triples of distinct integers with $1 \leqs i,j,k \leqs
n-1$. These linear combinations appear, with what we believe is a
type error, in \cite[Equation~(21)]{bhargava2008higher}.  The
remaining columns are explained below.
\begin{table}[ht]
\caption{Structure constants for rings of rank $n$ (up to shear)}
\centering
$\begin{array}{l|c|c|c}
& n=4 & n=5 & \text{$n \geqs 4$} \\ \hline
c^{k}_{ij} &\pm \{jjii\} & \pm \{ii \ell jj\} & 
  \pm \{i,i,1,2,\ldots \widehat{i},\ldots,
  \widehat{j},\ldots, \widehat{k},\ldots,n-1,j,j\} \\
c^{j}_{ii} &\pm \{iiik\} & \pm \{\ell iiik\} & \pm %(-1)^j
  \{i,i,1,\ldots,\widehat{i}, \ldots, \widehat{j},\ldots,n-1,i \} \\
c^{j}_{ij}-c^{k}_{ik} & \pm \{iijk\} & \pm \{jk \ell ii\} &  
  \pm \{i,i,1,\ldots,\widehat{i},\ldots,\widehat{j},\ldots,\widehat{k},
  \ldots,n-1,j,k\} \\
c^{i}_{ii}-c^{j}_{ij}-c^{k}_{ik} & \pm \{ikij\} &
 \pm \{ij \ell ki\} & \pm \{i,j,1,\ldots,\widehat{i},\ldots,
  \widehat{j},\ldots, \widehat{k},\ldots,n-1,k,i\}
\end{array}$
\label{table:strconst}
\end{table}

Let $F_{\bu}$ be a minimal free resolution of a set of $n$ points in
general position, with differentials $\phi_1, \ldots, \phi_{n-2}$
represented by matrices of linear and quadratic forms.  We write
\[ \phi_1 = \sum_{i \leqs j} P(i,j) x_i x_j \quad \text{ and }
  \quad \phi_{n-2} = \sum_{i \leqs j} Q(i,j) x_i x_j, \] where the
$P(i,j)$ are row vectors, and the $Q(i,j)$ are column vectors. We then
define
\[ \{a_1a_2\ldots a_n\}:=P(a_1,a_2) \frac{\partial \phi_{2}}{\partial
    x_{a_3}} \cdot \cdot \cdot \frac{\partial \phi_{n-3}}{\partial
    x_{a_{n-2}}} Q(a_{n-1},a_n). \]

When $n=4$ the minimal free resolution $F_\bullet$ takes the
form~\eqref{res:n=4} where $A$ and $B$ are ternary quadratic
forms. The symbols $\{ijk \ell\}$ were denoted $\lambda^{ij}_{k \ell}$
in \cite[Section~3.2]{bhargava2004higher}. The structure constants in
{\em loc. cit.}, up to shear, are then as recorded in
Table~\ref{table:strconst}, where $\pm$ denotes the sign of the
permutation taking $1,2,3$ to $i,j,k$.

When $n=5$ the structure theorem of Buchsbaum and Eisenbud
\cite{buchsbaum1982gorenstein} for Gorenstein ideals of codimension 3
shows that the minimal free resolution takes the form
\[ 0 \ra R(-5) \stackrel{P^T}{\ra} R(-3)^5 \stackrel{\Phi}{\ra}
  R(-2)^5 \stackrel{P}{\ra} R \ra 0, \] where $\Phi$ is a $5 \times 5$
alternating matrix of linear forms, and $P$ is the (signed) row vector
of $4 \times 4$ Pfaffians of $\Phi$.  Our symbols $\{ijk \ell m\}$
differ only by some factors of $2$ from those defined in
\cite[Section~4]{bhargava2008higher}.  The structure constants in {\em
  loc. cit.}, up to shear, are again as recorded in
Table~\ref{table:strconst}, where $\pm$ denotes the sign of the
permutation taking $1,2,3,4$ to $i,j,k, \ell$.

The expressions we give in the right hand column of
Table~\ref{table:strconst} are new.

\begin{theorem}
\label{thm:overZ}
Let $n \geqs 4$ be an integer. Then the structure constants
\begin{equation*}
  c_{ij}^{k} = \frac{1}{2n} 
  \frac{ \partial^2 \Omega_k}{\partial x_i \partial x_j}
\end{equation*}
satisfy the system of equations in Table~\ref{table:strconst}.
\end{theorem}
\begin{proof}
  It is clear from the definition that the symbol $\{ \cdots \}$ does
  not depend on the order of its first two arguments, or the order of
  its last two arguments.

  If $n=4$ then by~\eqref{res:n=4} we have $[i,j] = -[j,i]$ and
  $\{ijk \ell\} = -\{ k \ell i j\}$.  Let $i,j,k$ be an even
  permutation of $1,2,3$. By Definition~\ref{def:brackets} we have
  $\Omega_k = -2[i,j]$.  Using the product rule we compute
\begin{align*}
\frac{\partial^2 [j,i]}{\partial x_i \partial x_j} 
&= 4\{ jjii \} + \{ijij\} = 4\{ jjii \}, \\
\frac{\partial^2 [i,k]}{\partial x_i^2}  
& =  2\{iiik\} + 2\{iiik\} = 4\{ iiik \}, \\
\frac{\partial^2 [i,k]}{\partial x_i \partial x_j} 
+ \frac{\partial^2 [i,j]}{\partial x_i \partial x_k} 
&= 2\{iijk\} + \{ijik\} + 2\{iijk\} + \{ikij\}
= 4\{iijk\},  \\
 \frac{\partial^2 [k,j]}{\partial x_i^2}
 + \frac{\partial^2 [k,i]}{\partial x_i \partial x_j}
 + \frac{\partial^2 [i,j]}{\partial x_i \partial x_k}  
 &= 2\{ikij\} + \{ikij\} + 2\{jkii\} + 2\{iijk\} + \{ikij\}
 = 4 \{ikij\}.
 \end{align*}
 This proves the theorem in the case $n=4$.  We give the proof for
 $n \geqs 5$, and specify the correct choice of signs $\pm$, in
 Theorem~\ref{finaltheorem}.  It may also be checked, using
 Lemma~\ref{lem:sym}, that the expressions in
 Table~\ref{table:strconst} for $n \geqs 4$ do indeed specialise
 to those in the previous two columns when $n=4$ and $n=5$.
\end{proof}

We have now checked that Theorem~\ref{algebra constructed from n
  points} agrees, up to a shear, with the previously known
constructions for $n=3,4,5$. However the structure constants do not
agree exactly, since in Theorem~\ref{algebra constructed from n
  points} the basis elements $\alpha_1,\ldots,\alpha_{n-1}$ are
normalised (up to shear) so that they have trace zero, whereas the
algebras in \cite{delone1964theory, bhargava2004higher, bhargava2008higher}
are normalised so that
\[\begin{array}{ccl}
    n = 3 && c_{12}^1 = c_{12}^2 = 0, \\
    n = 4 && c_{12}^1 = c_{12}^2 = c_{13}^1 = 0, \\
    n = 5 && c_{12}^1= c_{12}^2 = c_{34}^3 = c_{34}^4 = 0. 
  \end{array}\]
In general we could normalise our basis by choosing a convention
such as
\begin{equation} \label{normalise:cyclic}
c_{12}^2 = c_{23}^3 = c_{34}^4 = \ldots = c_{n-2,n-1}^{n-1} = c_{n-1,1}^{1} = 0, 
\end{equation}
or when $n$ is odd 
\begin{equation} \label{normalise:pairwise} c_{12}^1= c_{12}^2 =
  c_{34}^3 = c_{34}^4 = \ldots = c_{n-2,n-1}^{n-2} = c_{n-2,n-1}^{n-1}
  = 0. \end{equation} With either convention, it is clear that there
is a unique way to modify our basis by a shear so that it satisfies
the convention.

Compared to normalising $\alpha_1, \ldots, \alpha_{n-1}$ to have trace
zero, these conventions break symmetry, but have the advantage of
working in all characteristics.  They are also useful for constructing
orders in number fields as we now explain.

We take $K = \Q$ and suppose that the minimal free resolution
$F_{\bullet}$ has integer coefficients, i.e. the differentials
$\phi_k$ are represented by matrices of polynomials in
$\Z[x_1, \ldots, x_{n-1}]$. By Definition~\ref{def:brackets} the
structure constants for the $\Q$-algebra $A$ in Theorem~\ref{algebra
  constructed from n points} are integral, and so determine an order
$B \subset A$. This ring decomposes (as a $\Z$-module) as
$B = \Z \oplus B_0$ where $B_0$ is the subset of elements of trace
zero. Therefore $\Tr_{A/\Z}(B) = n\Z$, and so any prime dividing $n$
necessarily ramifies in $B$.  More generally, this order can never be
maximal at the primes dividing $2n$.  Indeed, if we choose our
structure constants using Table~\ref{table:strconst} and one of the
normalisation conventions~\eqref{normalise:cyclic}
or~\eqref{normalise:pairwise}, then these define an order $B'$ with
$B \subset B' \subset A$. From the factor $2n$ in the statement of
Theorem~\ref{thm:overZ} we see that the index of $B$ in $B'$ is
$(2n)^{n-1}$.  For $n=3,4,5$, the ring $B'$ is the one constructed in
\cite{delone1964theory, bhargava2004higher, bhargava2008higher}.
For general $n$, working with $B'$ rather
than $B$ gives a larger order with smaller discriminant, and hence a
sharper bound in~\cite[Theorem~1.0.1]{LazarThesis}.

\begin{remark} We briefly mention three respects in which the theory
  for $n=3,4,5$, as described in \cite{ggs, bhargava2004higher,
  bhargava2008higher}, is still more
  developed than that for general $n$.
\begin{enumerate}
\item We conjecture than any order in an \'etale $\Q$-algebra of rank
  $n$ is necessarily of the form $B'$ for some minimal free resolution
  $F_\bullet$ with integer coefficients. This is known for
  $n=3,4,5$. It is also true for the ring
  $\Z \times \Z \times \ldots \times \Z$ by the calculations in
  Sections~\ref{structure constants section} and~\ref{wilson
    description section}.  We hope to investigate this conjecture
  further in future work.
\item For $n=3,4,5$ the binary cubics, pairs of ternary quadratics, and
  alternating matrices of linear forms, parameterise all rings of rank
  $n$, including degenerate rings such as $\Z[x]/(x^n)$.  It would be
  interesting to determine if there is a suitable class of
  ``degenerate'' minimal free resolutions for $n >5$ that correspond
  to these rings.
\item The results for $n=3,4,5$ have been used by Davenport and
  Heilbronn \cite{davenport1971density} and Bhargava
  \cite{bhargava2005density, bhargava2010density} to give an
  asymptotic count of number fields of degree $n$ ordered by
  discriminant. Since we do not have a description of the space of
  minimal free resolutions that lends itself to the counting arguments
  used in the geometry of numbers, it remains a difficult problem to
  extend these results to $n > 5$.
\end{enumerate}
\end{remark}

\section{Points in general position and étale algebras}
\label{sec:bijection}

In this section we give a slightly different perspective on the
classical fact that there is an equivalence of categories between the
category of finite sets of points with a continuous action of the
absolute Galois group, and the category of finite dimensional étale
algebras. These ideas feature prominently in works of Bhargava, see
especially the discussion in \cite[page~59]{bhargava2008higher}, as well
as the work of his students Wood \cite{woodrings} and Wilson
\cite{wilson}.

We fix an integer $n \geqs 3$.  Let
$\mc{X}=\{X \subset \PP^{n-2} : \ X \ \text{is a set of} \ n \
\text{points in general position}\}$. Note that $\mc{X}(K)$ consists
of sets $X$ which are defined over $K$, viewed as zero-dimensional
varieties, but the individual (geometric) points of $X$ need not be
defined over $K$.  The group $\mathrm{PGL}_{n-1}(K)$ acts on
$\mc{X}(K)$ by changes of coordinates.
	
An $n$-dimensional commutative $K$-algebra $A$ is {\em non-degenerate}
if the trace form associated to $A$ is non-degenerate. This is
equivalent to requiring that $A$ is étale over $K$, or that there
exists an isomorphism $A \otimes_{K} \bar{K} \cong \bar{K}^n$.  For
example, the ring $A=\Gamma(X,\mc{O}_X)$ of global functions on a set
of $n$ points $X \in \mc{X}(K)$ is a non-degenerate $K$-algebra of
dimension $n$.

The following fact seems to be well-known, but we could not find an
adequate reference.

\begin{proposition} \label{algebras and points lemma} Let $\mc{A}$ be
  the set of isomorphism classes of non-degenerate $n$-dimensional
  $K$-algebras. Then the map $X \mapsto \Gamma(X,\mc{O}_{X})$ induces
  a bijection between the set of $\mathrm{PGL}_{n-1}(K)$-orbits of
  $\mc{X}(K)$ and the set $\mc{A}$.
\end{proposition}	

Any two elements of $\mc{X}(K)$ that lie in the same
$\mathrm{PGL}_{n-1}(K)$-orbit have isomorphic rings of global
functions, and so map to the same element of $\mc{A}$. Therefore the
map in Proposition~\ref{algebras and points lemma} is well defined. We
must show it is a bijection. First we need two lemmas.

\begin{lemma} \label{projective trivial}
  \begin{enumerate} \item Let $A$ be a
  non-degenerate $n$-dimensional $K$-algebra.  Let $M$ be a locally
  free $A$-module of rank 1. Then $M$ is free,
  i.e., it is isomorphic to $A$ as an $A$-module.  
  \item Let $X \subset \PP^{n-2}$ be a set of $n$ distinct points, and
  let $A = \Gamma(X,\mc{O}_{X})$. Then $X$ is the image of a map
  $\Spec A \to \PP^{n-2}$ given by
  $(\alpha_1 : \ldots : \alpha_{n-1})$ for some
  $\alpha_1, \ldots, \alpha_{n-1} \in A$.
  \end{enumerate}
\end{lemma}

\begin{proof}
  (i) Since $A$ is non-degenerate it is isomorphic as an algebra to a
  direct product of fields, say,
  $A \cong A_1 \times \cdots \times A_k$. For each $1\leqs i\leqs k$,
  let $e_i$ be the idempotent corresponding to the factor $A_i$, so
  that $\sum^{k}_{i=1} e_i=1$, $Ae_i \cong A_i$ as an $A$-module,
  $e_i^2=e_i$ and $e_ie_j=0$ for all $i \not= j$.  Then we have the
  decomposition $M=e_1M\oplus e_2M \oplus \cdots \oplus e_kM$, where
  each module $e_iM$ is an $A_i$-vector space. Since $M$ is locally
  free of rank 1, each $e_iM$ is 1-dimensional, and so we may choose a
  basis vector $f_i$. Then $M=Af \cong A$, where $f=\sum^{k}_{i=1} f_i$. \\
  (ii) The embedding of $X = \Spec A$ in $\PP^{n-2}$ is determined by
  global sections $\ell_1, \ldots, \ell_{n-1}$ belonging to the
  $A$-module $M = \Gamma(X, \mc{O}_X(1))$. We see by (i) that $M$ is a
  free $A$-module of rank 1, say generated by $m \in M$.  We then
  write $\ell_i = \alpha_i m$ for some $\alpha_i \in A$.
\end{proof}

\begin{lemma} \label{unit-genpos} Let $A$ be a non-degenerate
  $n$-dimensional $K$-algebra.  Let $X$ be the image of the map
  $\Spec A \to \PP^{n-2}$ given by
  $(\alpha_1 : \ldots : \alpha_{n-1})$ for some
  $\alpha_1, \ldots, \alpha_{n-1} \in A$. Then $X$ is a set of $n$
  points in general position if and only if there exists a unit
  $\lambda \in A^\times$ such that
  $\lambda \alpha_1, \ldots, \lambda \alpha_{n-1}$ is a $K$-basis for
  the trace zero subspace of~$A$.  Moreover, if such a $\lambda$
  exists then it is unique up to multiplication by an element of
  $K^\times$.
\end{lemma}

\begin{proof}
  We may assume that $\alpha_1, \ldots, \alpha_{n-1}$ are linearly
  independent over $K$, since otherwise it is clear that neither
  condition is satisfied.

  The uniqueness follows from the non-degeneracy of the trace
  form. For existence, it is clear by linear algebra over $K$ that
  there exists non-zero $\lambda \in A$ such that
  $\Tr_{A/K} (\lambda \alpha_j) = 0$ for all $1 \leqs j \leqs n-1$. It
  remains to show that $\lambda$ is a unit.

  Since $A$ is isomorphic to a product of fields, we may write any
  element of $A$ as a unit times an idempotent. It therefore suffices
  to consider $\lambda$ an idempotent.  Writing
  $\sigma_1, \ldots, \sigma_n$ for the distinct $K$-algebra
  homomorphisms $A \to \bar{K}$ we have
  \[ X(\bar{K}) = \{ (\sigma_i(\alpha_1) : \ldots :
    \sigma_i(\alpha_{n-1})) : 1 \leqs i \leqs n \} \subset
    \PP^{n-2}(\bar{K}). \] Since $\lambda$ is an idempotent we may
  order the $\sigma_i$ such that
  $(\sigma_1(\lambda), \ldots, \sigma_n(\lambda)) = (1, \ldots, 1,0,
  \ldots 0)$ where there are (say) $m$ ones and $n-m$ zeros.  Then
  $\sum_{i=1}^m \sigma_i(\alpha_j) = \Tr_{A/K}(\lambda \alpha_j) =0$
  for all $1 \leqs j \leqs n-1$. Since $X$ is in general position, this
  forces $m=n$, and so $\lambda=1$ as required.
\end{proof}

\medskip

\noindent
{\em Proof of Proposition~\ref{algebras and points lemma}.}
First, to prove surjectivity, we suppose that $A$ is a non-degenerate
$n$-dimensional $K$-algebra.  Then we pick
$\alpha_1, \ldots, \alpha_{n-1}$ a $K$-basis for the trace zero
subspace of $A$, and let $X$ be the image of the map
$\Spec A \to \PP^{n-2}$ given by $(\alpha_1 : \ldots : \alpha_{n-1})$.
Lemma~\ref{unit-genpos} shows that $X$ is in general position, and we
then have $X \mapsto A$.  Next, to prove injectivity, we suppose that
$X_1$ and $X_2$ both map to $A$.  By Lemmas~\ref{projective trivial}
and~\ref{unit-genpos} both $X_1$ and $X_2$ are embedded in $\PP^{n-2}$
using (possibly) different choices of bases for the trace zero
subspace of $A$.  Hence there exists an element of
$\mathrm{PGL}_{n-1}(K)$ taking $X_1$ to $X_2$.  \qed

\section{Overview of the proof}
\label{sec:overview}

Let $V=\langle x_1,\ldots,x_{n-1} \rangle$ be the space of linear
forms on $\mathbb{P}^{n-2}$, and denote the dual basis of $V^{*}$ by
$x^*_1,\ldots,x_{n-1}^{*}$. The quadratic forms
$\Omega_1, \ldots, \Omega_{n-1}$ determine an element $\Omega$ in
$V^* \otimes S^2 V $ via the formula
\begin{equation}
  \label{def:Omega}
  \Omega := \sum_{j=1}^{n-1} x^*_{j} \otimes \Omega_{j}.
\end{equation}

The following proposition, which we prove in Section~\ref{sec:change},
shows that the construction of $\Omega$ from the minimal free
resolution is invariant under change of basis of $V$.

\begin{proposition} \label{Omega quad change of coordinates} Let
  $F_{\bullet}$ be a minimal free resolution for a set
  $X \subset \PP^{n-2}$ of $n$ points in general position. We write
  $x_1, \ldots, x_{n-1}$ for our coordinates on $\PP^{n-2}$.  Let
  $x'_j=\sum_{i=1}^{n-1} g_{ij}x_i$ for some
  $g=(g_{ij}) \in \mathrm{GL}_{n-1}$. Writing
  $\phi_1, \ldots, \phi_{n-2}$ for the matrices representing the maps
  in the resolution $F_{\bullet}$, let $F'_{\bullet}$ be the
  resolution whose maps are given by matrices
  $\phi'_1,\ldots,\phi'_{n-2}$ where
  \[ \phi'_r(x_1,\ldots,x_{n-1})=\phi_r(x'_1,\ldots,x'_{n-1}). \] Let
  $\Omega$ and $\Omega'$ be the elements of $V^* \otimes S^2 V $
  associated to the resolutions $F_{\bullet}$ and $F'_{\bullet}$
  respectively. Then we have
  \[ \Omega'=(\det g) (g \cdot \Omega), \] where the action of
  $g$ on $\Omega$ is the standard action of $\mathrm{GL}_{n-1}$ on
  $ V^* \otimes S^2 V $.
\end{proposition}

Our main theorem (Theorem~\ref{algebra constructed from n points})
gives an expression for the structure constants of the algebra $A$ in
terms of the quadratic forms $\Omega_1, \ldots, \Omega_{n-1}$
determined by a minimal free resolution of $X$. In fact we prove the
following strengthening of that theorem.

\begin{theorem}\label{structure constants theorem}
  Let $X \subset \PP^{n-2}$ be a set of $n$ points in general
  position. Let $A = \Gamma(X,\mc{O}_X)$ be the coordinate ring of $X$
  and let $V$ be the trace zero subspace of $A$.
  \begin{enumerate}
  \item There is a $K$-basis $\alpha_1, \ldots, \alpha_{n-1}$ of $V$,
    unique up to multiplication by an overall scalar, such that the
    embedding of $X = \Spec A$ in $\PP^{n-2}$ is given by
    $(\alpha_1: \ldots : \alpha_{n-1})$.
  \item Let $\alpha_1, \ldots, \alpha_{n-1}$ be as in (i) and let
    $\alpha_0^*, \alpha_1^*, \ldots, \alpha_n^*$ be the basis for $A$
    that is dual to $1, \alpha_1, \ldots, \alpha_{n-1}$ with
    respect to the trace pairing $(x,y) \mapsto \Tr_{A/K}(xy)$.  If
    $\Omega_1, \ldots, \Omega_{n-1}$ are the quadratic forms
    determined by a minimal free resolution of $X$ then there exist
    constants $\lambda, c^{0}_{ij} \in K$ such that
    \begin{equation}
      \label{eqn:sc}
      \alpha^*_i \alpha^*_j=c^{0}_{ij}+ \lambda \sum^{n-1}_{k=1} %\lambda
      \frac{\partial^2 \Omega_k}{\partial x_i \partial x_j} \alpha^*_k,
    \end{equation}
    for all $1 \leqs i,j \leqs n-1$.
  \end{enumerate}
\end{theorem}

\begin{proof}
  (i) This follows from Lemmas~\ref{projective trivial}
  and~\ref{unit-genpos}. \\
  (ii) We claim we are free to make changes of coordinates on
  $\PP^{n-2}$. This is proved by considering the effect of such a
  change of coordinates on each term in \eqref{eqn:sc}.  We may
  organise this calculation as follows.  First we use the trace
  pairing to identify $A=(K \cdot 1) \oplus V^{*}$. Then
  multiplication in $A$ determines a symmetric bilinear map
  $V^{*} \times V^{*} \xrightarrow{} V^{*}$ and hence an element of
  $V^{*} \otimes S^2 V$. Next we use part (i) of the theorem to
  identify $V$ with the space of linear forms on $\PP^{n-2}$.  The
  quadratic forms $\Omega_1 ,\ldots, \Omega_{n-1}$ determine an
  element $\Omega$ in $V^{*} \otimes S^2 V$ via \eqref{def:Omega}. The
  theorem asserts that these two elements of $V^{*} \otimes S^2 V$ are
  equal, up to multiplication by a scalar. To prove our claim it
  suffices to show that these two elements transform, under a change
  of coordinates, according to the natural action of $\GL(V)$ on
  $V^{*} \otimes S^2 V$. In the first case this is clear from the
  construction, and in the second case we use Proposition~\ref{Omega
    quad change of coordinates}.

  We are also free to extend our base field $K$, and so may assume by
  a change of coordinates that $X$ is the standard set of $n$ points
  in general position given by
  $P_1=(1:0\ldots:0),P_2=(0:1:0\ldots:0),\ldots,P_{n-1}=(0:\ldots:0:1)$
  and $P_{n}=(1:1:\ldots:1)$.  Having reduced to this special case, we
  identify $A \cong K^n$ via
  $\alpha \mapsto (\alpha(P_1), \ldots, \alpha(P_n))$. Then the
  $K$-basis for the trace zero subspace as determined in part (i) of
  the theorem is
  \begin{equation*}
    \begin{split}
      \alpha_1&=\frac{1}{n}(1,0,0,\ldots,0,-1), \\
      \alpha_2&=\frac{1}{n}(0,1,0,\ldots,0,-1),\\
      & \hspace{4em} \vdots\\
      \alpha_{n-1}&=\frac{1}{n}(0,0,0,\ldots,1,-1),
    \end{split}
  \end{equation*}
  where at this stage the overall scaling by a factor $1/n$ is
  arbitrary, but has been chosen to simplify the calculations that
  follow. The multiplication on $K^n$ is given by multiplication in
  each component separately, and the trace pairing on $K^n$ is given
  by the standard dot product. Following the statement of part (ii) of
  the theorem, we compute:
  \begin{equation*}
    \begin{split}
      \alpha^*_1&=(n-1,-1,-1,\ldots,-1), \\
      \alpha^*_2&=(-1,n-1,-1,\ldots,-1),\\
      & \hspace{5em} \vdots\\
      \alpha^*_{n-1}&=(-1,-1,\ldots,n-1,-1).
    \end{split}
  \end{equation*}
  Then for $1 \leqs i,j \leqs n-1$, the multiplication is given by
  \[ \alpha^*_i\alpha^*_j = \left\{ \begin{array}{ll}
  -1 - \alpha^*_i - \alpha^*_j & \text{ if } i \ne j, \\ (n-1)
  + (n-2) \alpha_i^* & \text{ if } i = j. \end{array} \right.\]
  Equation~\eqref{eqn:sc} with $\lambda = 1/2$ then follows from the
  next lemma.
\end{proof}

\begin{lemma} \label{quadric coomputation lemma} Let
  $X \subset \PP^{n-2}$ be the standard set of $n$ points in general
  position. Then the quadratic forms $\Omega_i$ are given, up to an
  overall scalar, by $\Omega_i=n x_i^2-2x_i \sum^{n-1}_{j=1}x_j$ for
  all $1\leqs i \leqs n-1$.
\end{lemma}

We checked this lemma for all $n \leqs 10$ by computer algebra, and
could easily extend this to some larger values of $n$.  In
Sections~\ref{structure constants section} and~\ref{wilson description
  section} we give a proof that is valid for all $n$.

\section{Some symmetries}
\label{sec:sym}
We prove some symmetries satisfied by the square bracket and double
square bracket symbols (see Definition~\ref{def:brackets}).  For
convenience in this section we put $m=n-2$, since all our resolutions
have length $m$.

\begin{lemma} \label{swaplemma} If $2\leqs r \leqs m-2$ then
  \[ [a_1,\ldots,a_{r},a_{r+1},\ldots,a_{m}] =
     -[a_1,\ldots,a_{r+1},a_{r},\ldots,a_{m}]. \]
\end{lemma}
\begin{proof} For this range of $r$, both $\phi_r$ and $\phi_{r+1}$
  are matrices of linear forms. We differentiate the relation
  $\phi_r \phi_{r+1}=0$. By the Leibniz rule,
  \[ 0=\frac{\partial^2 (\phi_r \phi_{r+1})}{\partial x_{a_r} \partial
      x_{a_{r+1}}}=\frac{\partial \phi_r}{\partial
      x_{a_r}}\frac{\partial \phi_{r+1}}{\partial
      x_{a_{r+1}}}+\frac{\partial \phi_r}{\partial
      x_{a_{r+1}}}\frac{\partial \phi_{r+1}}{\partial x_{a_r}}, \]
  hence the desired relation.
\end{proof}

\begin{remark} The lemma holds only for $2\leqs r \leqs m-2$. Indeed,
  swapping the first two $a$'s or the last two $a$'s need not simply
  result in a sign change. See Lemma~\ref{single bracket points} for
  an explicit example.  We introduce the symbols
  $[[a_1,a_2,\ldots,a_{m}]]$ to rectify this; see Lemma~\ref{double
    bracket permute} below.
\end{remark} 

\begin{lemma} \label{reverselemma} We have
  $[a_1,a_2,\ldots,a_{m-1},a_{m}]=\pm[a_{m},a_{m-1},\ldots,a_2,a_1]$,
  where the sign is $+1$ if $m \equiv 0,1 \ (\mathrm{mod} \ 4)$ and
  $-1$ if $m \equiv 2,3 \ (\mathrm{mod} \ 4)$.
\end{lemma}
\begin{proof}
  The dual of an $R$-module $M$ is $M^* = \operatorname{Hom}_R(M,R)$.
  If an $R$-module map $M \to N$ between free $R$-modules is
  represented by a matrix $\phi$ (with respect to some bases) then the
  dual map $N^* \to M^*$ is represented by the transpose matrix
  $\phi^T$ (with respect to the dual bases).
	    
  We saw in Theorem~\ref{min res theorem} that $X$ is Gorenstein. This
  implies that its minimal free resolution $F_\bullet$ is self-dual.
  Explicitly, there is a commutative diagram
  \[ \xymatrix{ 0 \ar[r] & R \ar[r]^-{\phi_m} \ar[d]_{\pm 1}
    & F_{m-1} \ar[r] \ar[d] & \ldots \ar[r] & F_1 \ar[r]^-{\phi_1} \ar[d]
    & R \ar[r] \ar@{=}[d] & 0 \\ 0 \ar[r] & R \ar[r]^-{\phi_1^T}
    & F_1^* \ar[r] & \ldots \ar[r] & F_{m-1}^* \ar[r]^-{\phi_m^T}
    & R \ar[r] & 0 }\] where the vertical maps are isomorphisms, 
   the right most one is the identity map, and the left most one
   is multiplication by $\pm 1$.  According to
   \cite[page~123]{bruns1998cohen}, the sign is
  $+$ if and only if $m \equiv 0,1 \pmod{4}$.  The lemma now
  follows from the definition of the square brackets notation.
\end{proof}
\begin{lemma} \label{swaplemma2} If the terms indicated by $\ldots$
  are the same in each case then
  \begin{enumerate}
  \item $[a,\ldots,b] + [b,\ldots,a] = 0$, and
  \item
    $[a,b,\ldots,c,d]+[b,a,\ldots,c,d]+[a,b,\ldots,d,c]+[b,a,\ldots,d,c]=0$.
  \end{enumerate}
\end{lemma}
\begin{proof}
  Part (i) follows from Lemmas~\ref{swaplemma}
  and~\ref{reverselemma}. For the second part, as $\phi_1\phi_2=0$ and
  $\phi_{m-1}\phi_{m}=0$, we have
  \[ 0=\frac{\partial^2 (\phi_1 \phi_{2})}{\partial x_{a} \partial
      x_{b}}=\frac{\partial \phi_1}{\partial x_{a}}\frac{\partial
      \phi_{2}}{\partial x_{b}}+\frac{\partial \phi_1}{\partial
      x_{b}}\frac{\partial \phi_{2}}{\partial x_{b}}+\frac{\partial^2
      \phi_1}{\partial x_a \partial x_b}\phi_2, \] and similarly,
  \[ 0=\frac{\partial^2 (\phi_{m-1} \phi_{m})}{\partial x_{c} \partial
      x_{d}}=\frac{\partial \phi_{m-1}}{\partial x_{c}}\frac{\partial
      \phi_{m}}{\partial x_{d}}+\frac{\partial \phi_{m-1}}{\partial
      x_{d}}\frac{\partial \phi_{m}}{\partial
      x_{c}}+\phi_{m-1}\frac{\partial^2 \phi_{m}}{\partial x_c
      \partial x_d}. \] Thus it suffices to show that
  \[ \frac{\partial^2 \phi_1}{\partial x_a \partial x_b} \phi_2
    \frac{\partial \phi_3}{\partial x_{a_3}}\ldots\frac{\partial
      \phi_{m-2}}{\partial x_{a_{m-2}}}\phi_{m-1}\frac{\partial^2
      \phi_{m}}{\partial x_c \partial x_d}=0. \] For
  $2 \leqs r \leqs m-3$ we have
  \[ \phi_r \frac{\partial \phi_{r+1}}{\partial x_{p}}=-\frac{\partial
      \phi_r}{\partial x_{p}}\phi_{r+1}. \] We use this relation to
  move the undifferentiated term to the right until we get an
  expression involving $\phi_{m-2}\phi_{m-1}$, which vanishes.
\end{proof}

As in Definition~\ref{def:brackets}, we let $\sigma=(123 \ldots m) \in S_m$
and define \[ [[a_1, \ldots, a_m ]] = \sum_{r=0}^{m-1} [
  a_{\sigma^{2r}(1)},\ldots, a_{\sigma^{2r}(m)} ]. \]
	
\begin{lemma} \label{double bracket permute} For $\tau \in S_{m}$ we
  have
  $[[a_{\tau(1)},\ldots,a_{\tau(m)}]]=\mathrm{sign}(\tau)[[a_1,\ldots,a_{m}]]$.
\end{lemma}
\begin{proof}
  For $m \leqs 3$ the lemma follows easily from
  Lemma~\ref{swaplemma2}(i), so we may assume $m \geqs 4$.
	    
  We first prove the lemma in the case $\tau=(12)$. We have
  \[ [[a_1,a_2,a_3, \ldots, a_m ]] = [ a_1, a_2, a_3, \ldots, a_m ] +
    [a_3,a_4, \ldots a_m,a_1, a_2 ] + \ldots \] We consider the effect
  of switching $a_1$ and $a_2$ on each term on the right. For the
  terms we have not written out, the answer is that they change sign,
  and indeed this follows from Lemmas~\ref{swaplemma}
  and~\ref{swaplemma2}(i), the latter being used for the term
  $[a_2,a_3, \ldots ,a_m, a_1]$ which only occurs if $n$ is odd.  We
  are left with the two terms we did write out.  We treat them
  together.  By repeatedly using Lemma~\ref{swaplemma} to move $a_3$
  to the right, we have
  \[ [a_2,a_1,a_3,\ldots,a_{m}]=(-1)^{m}[a_2,a_1,a_4,\ldots,a_{m-1},a_3,a_{m}], \]
  and similarly
  \[
    [a_3,a_4,\ldots,a_{m},a_2,a_1]=-[a_1,a_4,\ldots,a_{m},a_2,a_3]
    =(-1)^m[a_1,a_2,a_4,\ldots,a_{m-1},a_{m},a_3], \]
  where the first equality is a consequence
  of~Lemma~\ref{swaplemma2}(i) and the second is a repeated
  application of~Lemma~\ref{swaplemma}. By Lemma~\ref{swaplemma2}(ii)
  we have
  \begin{align*}
    (-1)^{m}([a_2,&a_1,a_4,\ldots,a_{m-1},a_3,a_{m}]
    + [a_1,a_2,a_4,\ldots,a_{m-1},a_{m},a_3]) \\
    & = (-1)^{m+1}([a_1,a_2,a_4,\ldots,a_{m-1},a_3,a_{m}]
    + [a_2,a_1,a_4,\ldots,a_{m-1},a_{m},a_3]).
  \end{align*}
  This is equal, by Lemma~\ref{swaplemma2}(i) and repeated
  application of Lemma~\ref{swaplemma}, to
  \[ -[a_1,a_2,a_3,\ldots,a_{m}]-[a_3,a_4,\ldots,a_m,a_1,a_2]. \] This
  completes the proof in the case $\tau=(12)$.
		
  It is immediate from the definition of $[[ \ldots ]]$
  % the double square brackets notation
  that the lemma holds with $\tau = \sigma^2$.  If $n$ is odd
  then $\sigma^2$ and $(12)$ generate $S_{m}$ and we are done. If $n$
  is even then $S_{m}$ is generated by $\sigma^2$, $(12)$ and
  $(23)$. So it suffices to prove the lemma for $\tau = (23)$. However
  the proof in this case goes through term by term using
  Lemmas~\ref{swaplemma} and~\ref{swaplemma2}(i).
\end{proof}

\begin{remark}
  If $n$ is odd then using $\sigma$ instead of $\sigma^2$ would make
  no difference in the definition of $[[\ldots]]$. However if $n$ is
  even then this would give an expression that is identically zero, as
  may be seen from Lemma~\ref{double bracket permute}, noting that
  $\sigma$ is an odd permutation.
\end{remark}

\section{Changes of coordinates}
\label{sec:change}
In this section we prove Proposition~\ref{Omega quad change of
  coordinates}.  Let $V = \langle x_1, \ldots, x_{n-1} \rangle$ be the
space of linear forms on $\PP^{n-2}$, and let
$x_1^*, \ldots, x_{n-1}^*$ be the dual basis for $V^*$. We identify
$S^d V$ with the space of degree $d$ homogeneous polynomials
$F$ in $K[x_1, \ldots, x_{n-1}]$. The natural left actions of
$\GL_{n-1}$ on $S^d V$ and on $V^*$ are given by
\begin{equation}
    \label{action_on_polys}
    (g \cdot F)(x_1, \ldots, x_{n-1}) = F(x'_1, \ldots, x'_{n-1}),
\end{equation} where $x'_j=\sum^{m}_{i=1}g_{ij}x_{i}$, and
\begin{equation}
    \label{action_on_dual}
    g \cdot x_j^* = \sum_{i=1}^{n-1} (g^{-T})_{ij} x_i^*. 
\end{equation}

In Section~\ref{sec:state} we defined quadratic forms
$[[\ldots]]$ and $\Omega_j$ associated to a minimal free resolution
with differentials $\phi_r$. Fix any $g \in \GL_{n-1}$.  We now write
$[[\ldots]]'$ and $\Omega_j'$ for the quadratic forms associated to
the minimal free resolution with differentials
$\phi'_r(x_1,\ldots,x_{m})=\phi_r(x'_1,\ldots,x'_m)$, where
$x'_j=\sum^{m}_{i=1}g_{ij}x_{i}$.  We must prove that
\begin{equation}
\label{eqn:change_coords}
\sum_{j=1}^{n-1} x^*_j \otimes \Omega'_j = (\det g)\sum_{j=1}^{n-1}
\left( g \cdot x^*_j\right) \otimes \left(g \cdot \Omega_j \right).
\end{equation}
It suffices to prove this claim for $g$ running over a set of
generators for the group $\GL_{n-1}$, and accordingly we consider $g$
a diagonal matrix, $g$ a permutation matrix, and $g$ a unipotent
matrix.

First we suppose that $g$ is diagonal, say with diagonal entries
$\lambda_1, \ldots, \lambda_{n-1}$.  By the chain rule, we have
\[ [[a_1,a_2,\ldots,a_{n-2}]]'(x_1, \ldots, x_{n-1}) =
  \lambda_{a_1}\cdots\lambda_{a_{n-2}}[[a_1,a_2,\ldots,a_{n-2}]]
  (x'_1, \ldots, x'_{n-1}). \]
Therefore 
\[ \Omega_j'(x_1,\ldots,x_{n-1})=(\det g) \lambda_j^{-1}      
   \Omega_j(x'_1,\ldots,x'_{n-1}), \]
and so~\eqref{eqn:change_coords} follows by~\eqref{action_on_polys}
and~\eqref{action_on_dual}.

We next suppose that $g$ is the permutation matrix corresponding to
the transposition $\tau=(ab)$ for some $1 \leqs a<b \leqs n-1$.  If
$j \notin \{a,b\}$ then by Lemma~\ref{double bracket permute} we have
\[  [[ \tau(1),\tau(2), \ldots, 
       \widehat{\tau(j)}, \ldots, \tau(n-1) ]] 
       = - [[ 1,2, \ldots, \widehat{j}, \ldots, n-1 ]], \] and
\begin{align*}	  
  [[ 1, \ldots, \widehat{a}, \ldots, b-1,a,b+1, \ldots, n-1 ]] 
        &=(-1)^{b-a-1}  [[ 1, \ldots, \widehat{b}, \ldots, n-1 ]], \\
  [[ 1, \ldots, a-1,b,a+1, \ldots, \widehat{b}, \ldots, n-1 ]] 
        &=(-1)^{b-a-1}  [[ 1, \ldots, \widehat{a}, \ldots, n-1 ]].
\end{align*}
Therefore
\[ \Omega'_j(x_1, \ldots,x_{n-1}) = \left\{ \begin{array}{ll}
    -\Omega_j(x'_1, \ldots, x'_{n-1})  & \text{ if } j \notin \{ a,b \}, \\
    -\Omega_b(x'_1, \ldots, x'_{n-1})  & \text{ if } j = a, \\
    -\Omega_a(x'_1, \ldots, x'_{n-1})  & \text{ if } j = b, 
\end{array} \right. \]
and so~\eqref{eqn:change_coords} follows by~\eqref{action_on_polys}
and~\eqref{action_on_dual}.

Finally, we consider the case $g=I+tE_{21}$, where $E_{ij}$ is the
$(n-1) \times (n-1)$ matrix with a $1$ in the $(i,j)$-place, and all
other entries $0$.  The matrices $\phi_r'$ are given by
\[ \phi'_r(x_1,\ldots,x_{n-1})=\phi_r(x_1+tx_2,x_2,\ldots,x_{n-1}). \]
So by the chain rule
\[  \frac{\partial \phi'_r}{\partial x_j}(x_1,\ldots,x_{n-1})
  = \left(\frac{\partial \phi_r}{\partial x_j}
  + \delta_{j2} t\frac{\partial \phi_r}{\partial x_1}\right)
  (x_1+tx_2,x_2,\ldots,x_{n-1}). \]
Therefore, if $2 \not\in \{a_1,\ldots,a_{n-2}\}$ then
\[ [[a_1,\ldots,a_{n-2}]]'(x_1,\ldots,x_{n-1})
  = [[a_1,\ldots,a_{n-2}]](x_1',\ldots,x_{n-1}'), \]
whereas if $a_k=2$ then
\begin{align*}
  [[a_1,\ldots& ,a_{n-2}]]'(x_1,\ldots,x_{n-1}) \\
  & = [[a_1,\ldots,a_{n-2}]](x_1',\ldots,x_{n-1}')
  + t[[a_1,\ldots,a_{k-1},1,a_{k+1},\ldots,a_{n-2}]](x_1',\ldots,x_{n-1}').
\end{align*}
We note that if $1 \in \{a_1,\ldots,a_{n-2}\}$ then the final term
$[[a_1,\ldots,a_{k-1},1,a_{k+1},\ldots,a_{n-2}]]$ vanishes by
Lemma~\ref{double bracket permute}, since 1 appears twice.  Then by
definition of the $\Omega_j$ we have
\[\Omega'_j(x_1,\ldots,x_{n-1})=\Omega_j(x'_1,\ldots,x'_{n-1})
    - \delta_{j1} t \Omega_2(x'_1,\ldots,x'_{n-1}), \]
 and hence
 \begin{equation*} \label{eq aim} \sum_{j=1}^{n-1} x^*_j \otimes
   \Omega'_j(x_1,\ldots,x_{n-1}) = \sum_{j=1}^{n-1} \left( x^*_j -
     \delta_{j2} t x^*_1 \right) \otimes
   \Omega_j(x'_1,\ldots,x'_{n-1}).
 \end{equation*}
 Now~\eqref{eqn:change_coords} follows by~\eqref{action_on_polys}
 and~\eqref{action_on_dual}. This completes the proof of
 Proposition~\ref{Omega quad change of coordinates}.

\section{Reduction to the key lemma}
\label{structure constants section}
	
In Section~\ref{sec:overview} we reduced the proof of
Theorem~\ref{algebra constructed from n points} to a verification in
the case of the standard set of $n$ points:
$P_1=(1:0\ldots:0),
P_2=(0:1:0\ldots:0),\ldots,P_{n-1}=(0:\ldots:0:1)$ and
$P_{n}=(1:1:\ldots:1)$.  This verification depends on the following
``key lemma'', the proof of which we postpone to the next section.

\begin{lemma}
\label{single bracket points}
Let $X \subset \PP^{n-2}$ be the standard set of $n$ points in general
position. Fix a choice of minimal free resolution for $X$.  Then there
exists a scalar $\lambda \in K$ such that whenever $a_1,\ldots,a_{n-2},b$ is
a permutation of $1,2,\ldots,n-1$ we have
\[  [a_1,\ldots,a_{n-2}]= \pm \lambda (x_{b}-x_{a_1}-x_{a_{n-2}})x_{b}, \]
where $\pm$ is the sign of the permutation, and the symbol $[\cdots]$
was defined in Section~\ref{sec:state}.
\end{lemma}

We use this lemma to compute the quadratic forms
$\Omega_1, \ldots, \Omega_{n-1}$ associated to $X$.

\medskip

\noindent
{\em Proof of Lemma~\ref{quadric coomputation lemma}}.
Recall that we wrote $\sigma$ for the $(n-2)$-cycle in $S_{n-2}$ with
$\sigma(1)=2, \sigma(2) =3, \ldots, \sigma(n-2) = 1$. Let
$a_1, \ldots, a_{n-2},b$ be as in the statement of Lemma~\ref{single
  bracket points}. By the definition of $[[\cdots]]$,
and the fact $\sigma^2$ is an even permutation, we compute
\begin{align*}
[[a_{1},a_{2},\ldots,a_{n-2}]]&=
\sum_{k=1}^{n-2} [a_{\sigma^{2k}(1)},a_{\sigma^{2k}(2)},\ldots,a_{\sigma^{2k}(n-2)}]\\&=
\lambda \sum_{k=1}^{n-2}   (x_b - x_{a_{\sigma^{2k}(1)}} - x_{a_{\sigma^{2k}(n-2)}})x_{b}\\
&= \lambda \left((n-2)x_b^2 - 2 x_b \sum_{j \ne b} x_j\right) \\
&= \lambda \left( n x_b^2-2x_b \sum^{n-1}_{j=1} x_j \right).   
\end{align*}
It follows by the definition of the $\Omega_i$ and Lemma~\ref{double
  bracket permute} that
\[ \Omega_i=(-1)^{i}[[1,\ldots,\widehat{i},\ldots,n-1]] = (-1)^{n-1}
  \lambda \left(n x_i^2-2x_i \sum^{n-1}_{j=1}x_j \right). \] Since we
are only computing the $\Omega_i$ up to an overall scalar, the factor
$(-1)^{n-1}\lambda$ may be ignored.  This completes the proof of
Lemma~\ref{quadric coomputation lemma}.  \qed

\section{Proof of the key lemma (=Lemma~\ref{single bracket points})}
\label{wilson description section}

As before, let $X \subset \PP^{n-2}$ be the standard set of $n$ points
in general position.  In this section we prove Lemma~\ref{single
  bracket points}.  Our approach is inspired by an explicit
description of the minimal free resolution of the set $X$, due to
Wilson \cite[Chapter~5]{wilson}.  In \cite[Section~4.5]{LazarThesis}
we gave a different proof of Lemma~\ref{single bracket points}, based
on the method of unprojection.

\begin{lemma} \label{ideal set pts} If $n \geqs 4$ then the ideal
  $I:=I(X) \subset K[x_1,\ldots,x_{n-1}]=R$ is generated by the
  quadratic forms $x_i(x_j-x_k)$ for $i,j,k \in \{1,2, \ldots,n-1\}$
  distinct. If $n=3$ then $I$ is generated by $x_1x_2(x_1-x_2)$.
\end{lemma}

\begin{proof}
  The case $n=3$ is obvious. The case $n \geqs 4$ is
  \cite[Lemma~146]{wilson}. The proof is a simple computation, since
  by Theorem~\ref{min res theorem} we already know that $I$ is
  generated by quadratic forms.
\end{proof}

For the rest of this section we assume that $n \geqs 5$, since the
somewhat degenerate cases $n=3,4$ are easy to handle by a direct
computation.

We fix a minimal graded free resolution $(F_{\bu},\phi)$ of
$I$.  The idea is to describe $F_{\bu}$ by splicing together Koszul
complexes. For each pair $J=(j,k)$, with
$j,k \in \{1,2, \ldots, n-1\}$ distinct, consider the ideal
$I^{J} \subset I$ generated by the set of quadratic forms
\[\{x_i(x_j-x_k) : i = 1,2, \ldots, \widehat{j}, \ldots, \widehat{k},
  \ldots, n- 1 \}.
\] 
As a graded $K[x_1,\ldots,x_{n-1}]$-module, $I^{J}$ is isomorphic to
the ideal generated by the linear forms
$x_1, \ldots, \widehat{x}_j, \ldots, \widehat{x}_k, \ldots, x_{n-1}$,
and so is resolved by a Koszul complex. We write this complex as
\begin{align*}	
  K^{J}_{\bu}: \quad 0 \ra \wedge^{n-3} E^J 
  \xrightarrow{d_{n-3}} \wedge^{n-4} E^J \ra 
  \ldots \ra \wedge^2 E^J \stackrel{d_2}{\ra} E^J \stackrel{d_1}{\ra} R.
\end{align*}
where $E^J$ is a free $R$-module of rank $n-3$ with basis
$e_1, \ldots, \widehat{e}_j, \ldots, \widehat{e}_k, \ldots, e_{n-1}$,
and the differentials $d_m$ are given for $m > 1$ by
\begin{equation}
\label{koszul diff}
d_m(e_{i_1} \wedge \ldots \wedge e_{i_m})
  = \sum^{m}_{\ell=1} (-1)^\ell x_{i_\ell} \cdot
  (e_{i_1} \wedge \ldots \wedge \widehat{e}_{i_\ell} \wedge \ldots \wedge e_{i_m})
\end{equation}
and for $m=1$ by $d_1(e_i)=x_i(x_j-x_k)$.

The inclusion $I^{J} \subset I$ induces a map of chain complexes
$K^{J}_{\bu} \xrightarrow{} F_{\bu}$, i.e. a commutative diagram
\[ \xymatrix{ & 0 \ar[r] & \wedge^{n-3} E^J \ar[r] \ar[d] & \ldots \ar[r]
& \wedge^2 E_J \ar[r]^-{d_2} \ar[d] & E^J \ar[r]^{d_1} \ar[d] & R \ar@{=}[d] \\
0 \ar[r] & F_{n-2} \ar[r]^-{\phi_{n-2}} & F_{n-3} \ar[r] & \ldots
\ar[r] & F_2 \ar[r]^{\phi_2} & F_1 \ar[r]^{\phi_1} & R } \]
With notation as in Section~\ref{sec:state}, we may equally write this as
\[ \xymatrix{ & 0 \ar[r] & R(-n+2) \ar[r] \ar[d] & \ldots \ar[r]
   & R(-3)^{a_2} \ar[r]^-{d_2} \ar[d] & R(-2)^{a_1} \ar[r]^-{d_1}
   \ar[d] & R \ar@{=}[d] \\  0 \ar[r] & R(-n) \ar[r]^-{\phi_{n-2}} &
   R(-n+2)^{b_{n-3}} \ar[r] & \ldots \ar[r] & R(-3)^{b_2}
   \ar[r]^-{\phi_2} & R(-2)^{b_1} \ar[r]^-{\phi_1} & R } \] where
$a_i = \binom{n-3}{i}$ and the $b_i$ are specified in
Theorem~\ref{min res theorem}.  In particular, all the differentials
$d_i$ and $\phi_i$ are represented by matrices of linear forms, except
for $d_1, \phi_1$ and $\phi_{n-2}$ which are represented by matrices
of quadratic forms.  The map of chain complexes
$K^{J}_{\bu} \xrightarrow{} F_{\bu}$, which by construction is unique
up to chain homotopy, is actually uniquely determined. This is because
any such chain homotopy respects the grading of the modules in the
resolutions, and hence must be zero.

We denote the image of $e_{i_1} \wedge \ldots \wedge e_{i_m}$ in
$F_{m}$ by the symbol $(i_1 \wedge \ldots \wedge i_m) \otimes (j,k)$.
It follows from equation~\eqref{koszul diff} that for
$2 \leqs m \leqs n-3$ we have
\begin{equation} \label{diff eq}
\phi_m((i_1 \wedge \ldots \wedge i_m) \otimes  (j,k))=\sum^{m}_{\ell=1} (-1)^\ell x_{i_\ell} \cdot (i_1 \wedge \ldots \wedge \widehat{i}_\ell \wedge \ldots  \wedge i_m) \otimes  (j,k). 
\end{equation}

\begin{lemma} \label{lem:Koszul relns}
\begin{enumerate} 
\item For any $i_1, \ldots, i_m,j,k \in \{1,2, \ldots, n-1\}$ distinct we have
\[ (i_1 \wedge \cdots \wedge i_m) \otimes (j,k) 
 + (i_1 \wedge \cdots \wedge i_m) \otimes (k,j) = 0. \] 
\item For any $i_1, \ldots, i_m,j,k,\ell \in \{1,2, \ldots, n-1\}$
  distinct we have
\[ (i_1 \wedge \cdots \wedge i_m) \otimes (j,k) 
 + (i_1 \wedge \cdots \wedge i_m) \otimes (k,\ell) 
 + (i_1 \wedge \cdots \wedge i_m) \otimes (\ell,j) = 0. \] 
 \item The individual expressions $(i_1 \wedge \cdots \wedge i_m) \otimes (j,k)$
 are non-zero.
\end{enumerate}
\end{lemma} 

\begin{proof} (i) and (ii). The left hand side of each equation 
has degree $m+1$ in $F_m \isom R(-m-1)^{b_m}$. 
Since $\phi_m$ in injective in this degree, it suffices to check 
the image under $\phi_m$ is zero. The proof is now by induction on $m$. 
If $m=1$ then 
\begin{align*}
\phi_1 (i \otimes(j,k)+i \otimes(k,j)) &= x_i(x_j - x_k) + x_i(x_k - x_j) = 0, \\
\phi_1 (i \otimes(j,k) + i \otimes(k,\ell) + i \otimes(\ell,j)) 
        & = x_i(x_j - x_k) + x_i(x_k - x_\ell)+ x_i(x_\ell - x_j) = 0. 
\end{align*}
If $m > 1$ then we instead use~\eqref{diff eq} to give a linear combination
of $x_1, \ldots, x_{n-1}$ where each coefficient vanishes by the 
induction hypothesis. \\
(iii) This is proved by a similar, but easier, induction.
\end{proof}

We now give a formula for the differential
$\phi_{n-2} : F_{n-2} \xrightarrow{} F_{n-3}$.
\begin{lemma}
  The image of $\phi_{n-2}$ is generated as an $R$-module by
  \begin{equation} \label{final syzygy eq} t:=\sum_{j<k} x_{j}x_{k}
    \cdot (i_1 \wedge \ldots \wedge i_{n-3}) \otimes (j,k),
  \end{equation}
  where for each $j < k$ we pick $i_1,\ldots,i_{n-3},j,k$ an even
  permutation of $1,2, \ldots, n-1$.
\end{lemma} 

\begin{proof}
  We first note that $t$ has degree $n$ in
  $F_{n-3} \isom R(-n+2)^{b_{n-3}}$ and is non-zero by
  Lemma~\ref{lem:Koszul relns}(iii).  Since $F_{n-2} \cong R(-n)$ and
  $F_{\bu}$ is exact, it suffices to show that $t$ belongs to the
  kernel of $\phi_{n-3}$.  We find using \eqref{diff eq} that the
  coefficient of $x_1x_2x_3$ in $\phi_{n-3}(t)$ is
  % (up to a global sign change)
  \begin{align*}
    -(4 \wedge \ldots \wedge (n-1)) \otimes (1,2)
    - (4\wedge\ldots \wedge (n-1)) \otimes (2,3)
    - (4 \wedge \ldots \wedge (n-1)) \otimes (3,1),
  \end{align*}
  which vanishes by Lemma~\ref{lem:Koszul relns}(ii).  The same
  argument applies to the other coefficients.
\end{proof}

We now prove Lemma~\ref{single bracket points}. The symbol $[\ldots]$
was defined in terms of the partial derivatives of
$\phi_1, \ldots, \phi_{n-2}$, so we start by computing these.  As
$\phi_1(i \otimes (j,k))=x_i(x_j-x_k)$, we see that
\begin{equation} \label{deriv1}
\frac{\partial \phi_{1}}{\partial x_{i}}(i \otimes (j,k))=x_j-x_k, \quad
\frac{\partial \phi_{1}}{\partial x_{j}}(i \otimes (j,k))=x_i, \quad 
\frac{\partial \phi_{1}}{\partial x_{k}}(i \otimes (j,k))=-x_i.
\end{equation}
It is immediate from~\eqref{diff eq} that for $2 \leqs m \leqs n-3$ we
have
\begin{equation} \label{deriv2}
\frac{\partial \phi_{m}}{\partial x_{i}}((i_1 \wedge \ldots \wedge i_{m}) 
\otimes (j,k))= \left\{ \begin{array}{ll}
(-1)^{\ell} (i_1 \wedge \ldots \wedge \widehat{i_\ell} \wedge \ldots  \wedge i_m) \otimes  (j,k)  & \text{ if } i = i_\ell \\
0 & \text{ if } i \not\in \{i_1, \ldots, i_m \} \end{array} \right.
\end{equation}
Since the statement of Lemma~\ref{single bracket points} allows for
% the result we are proving is up to multiplication by
an overall scalar $\lambda \in K$, we may re-scale $\phi_{n-2}$ so that
$\phi_{n-2}(1) = t$ where $t$ is given by \eqref{final syzygy
  eq}. Then by Lemma~\ref{lem:Koszul relns}(i) we have
\begin{equation} \label{deriv3} \frac{\partial \phi_{n-2}}{\partial
    x_{k}}(1)=\sum_{\substack{j=1 \\ j \not= k}}^{n-1} x_{j} \cdot
  (i_1 \wedge \ldots \wedge i_{n-3}) \otimes (j,k)
\end{equation}
where for each $j$ we pick $i_1, \ldots, i_{n-3},j,k$ an even
permutation of $1,2, \ldots, n-1$.

Now let $a_1,\ldots,a_{n-2},b$ be a permutation of $1,2, \ldots, n-1$.
We seek to compute
\[ [a_1,a_2,\ldots,a_{n-2}]=\left( \frac{\partial \phi_{1}} {\partial
      x_{a_1}} \circ \frac{\partial \phi_{2}}{\partial x_{a_2}} \cdots
    \circ \frac{\partial \phi_{n-2}}{\partial
      x_{a_{n-2}}}\right)(1). \] By~\eqref{deriv3} we have
 \[ [a_1,a_2,\ldots,a_{n-2}]
   =\left( \frac{\partial \phi_{1}} {\partial x_{a_1}} \circ
     \frac{\partial \phi_{2}}{\partial x_{a_2}}\circ \cdots \circ
     \frac{\partial \phi_{n-3}}{\partial x_{a_{n-3}}}\right)\left(
     \sum_{\substack{j=1\\j\not=a_{n-2}}}^{n-1} x_{j} \cdot (i_1
     \wedge \ldots \wedge i_{n-3}) \otimes (j,a_{n-2})\right) . \]
 where for each $j$ we pick $i_1, \ldots, i_{n-3},j,a_{n-2}$ an even
 permutation of $1,2, \ldots, n-1$. It is clear by~\eqref{deriv2} that
 for a non-zero contribution we need
 $\{ a_2, \ldots, a_{n-3} \} \subset \{i_1, \ldots, i_{n-3} \}$,
 equivalently $\{a_1, a_{n-2}, b \} \supset \{j, a_{n-2} \}$. So the
 only terms to contribute to the sum are those with $j = a_1$ and
 $j = b$. Using~\eqref{deriv1} and~\eqref{deriv2} we compute
\begin{align*}
  [a_1,a_2,\ldots,a_{n-2}]&= \pm \frac{\partial \phi_{1}}{\partial x_{a_1}}
  \bigg( x_{b} \cdot a_{1} \otimes (b,a_{n-2})
  - x_{a_{1}} \cdot b \otimes (a_1,a_{n-2}) \bigg)\\
  & =\pm x_{b} ( x_{b} - x_{a_{1}} - x_{a_{n-2}} ).
\end{align*}
Finally, it may be checked that the sign $\pm$ only depends on $n$ and
the sign of the permutation sending $1,2, \ldots,n-1$ to
$a_1,a_2,\ldots,a_{n-2},b$.

This completes the proof of
Lemma~\ref{single bracket points}, and hence of
Theorem~\ref{algebra constructed from n points}.
 
\section{Proof of Theorem~\ref{thm:overZ}}
\label{sec:newproof}
 
In this section we prove Theorem~\ref{thm:overZ} for general
$n \geqs 4$. This extends the proof for $n = 4$ in
Section~\ref{sec:orders}, and is based on the proof for $n$ odd in
\cite[Section~3.8]{LazarThesis}.
 
We write $1, \ldots, \widehat{ijk}, \ldots, n-1$ for the sequence of
integers $1,2, \ldots, n-1$, with $i,j,k$ deleted (in whatever order
they occur). Let $\varepsilon_{ij}$ and $\varepsilon_{ijk}$ be the
signs of the permutations taking $1,2, \ldots, n-1$ to
$i,j,1, \ldots, \widehat{ij}, \ldots, n-1$ and
$i,j,k,1, \ldots, \widehat{ijk}, \ldots n-1$, respectively.

With notation as in Sections~\ref{sec:state} and~\ref{sec:orders}, we
prove the following theorem.  It is a refinement of
Theorem~\ref{thm:overZ}, in that we now specify the signs.

\begin{theorem} \label{finaltheorem} Let $1 \leqs i,j,k \leqs n-1$
  distinct.  Then
 \begin{align}
 \label{braceidentity1}
   \frac{\partial^2 \Omega_k}{\partial x_i \partial x_j}
   &= (-1)^{n+1} \varepsilon_{ijk} (2n)
   \{ i,i,1, \ldots, \widehat{ijk}, \ldots, n-1,j,j\}, \\
 \label{braceidentity2}
  \frac{\partial^2 \Omega_j}{\partial x_i^2} 
   &= \varepsilon_{ij} (2n)
   \{ i,i,1, \ldots, \widehat{ij}, \ldots, n-1,i\}, \\
 \label{braceidentity3}
 \frac{\partial^2 \Omega_j}{\partial x_i \partial x_j} - 
   \frac{\partial^2 \Omega_k}{\partial x_i \partial x_k} 
   &= (-1)^n \varepsilon_{ijk} (2n)
   \{i,i,1, \ldots, \widehat{ijk}, \ldots, n-1,j,k\}, \\
  \label{braceidentity4}
   \frac{\partial^2 \Omega_i}{\partial x_i^2}
   -\frac{\partial^2 \Omega_j}{\partial x_i \partial x_j} - 
   \frac{\partial^2 \Omega_k}{\partial x_i \partial x_k} 
   &= (-1)^{n+1} \varepsilon_{ijk} (2n)
   \{i,j,1, \ldots, \widehat{ijk}, \ldots, n-1,k,i\}. 
\end{align}
\end{theorem}   
 
For the proof we need some properties of the symbols $\{ \cdots \}$.
First, it is immediate from the definition that the symbol does not
depend on the order of the first two terms, or on the order of the
last two terms.  We have the following additional symmetry properties.

\begin{lemma}
\label{lem:sym}
\begin{enumerate}
\item For any $\tau \in S_{n-4}$ we have
  \[ \{ i,j, a_{\tau(1)}, \ldots, a_{\tau(n-4)},k,\ell \}
      = {\rm sign}(\tau) \{i,j,a_1, \ldots, a_{n-4},k,\ell\}. \]
\item If $i \in \{ a_1, \ldots, a_{n-3} \}$ then for any
  $\tau \in S_{n-3}$ we have
  \[ \{ i, a_{\tau(1)}, \ldots, a_{\tau(n-3)},k,\ell \}
    = {\rm sign}(\tau) \{i,a_1, \ldots, a_{n-3},k, \ell\}. \]
\item We have
  $\{ i,j, a_1, \ldots, a_{n-4}, k, \ell \}
  = -\{ k, \ell, a_1, \ldots, a_{n-4}, i,j \}$.
\end{enumerate}
\end{lemma}
\begin{proof}
  (i) In the case where $\tau$ is a transposition of consecutive
  elements this is proved exactly as in Lemma~\ref{swaplemma}.
  The general case follows. \\
  (ii) Differentiating $\phi_1 \phi_2 = 0$ gives
  $\{i,i,j, \ldots \} + \{i,j,i, \ldots \} = 0$. We are done by (i). \\
  (iii) Exactly as in Lemma~\ref{reverselemma}, we have
  $\{i,j,a_1, \ldots, a_{n-4},k, \ell\} = \pm \{k,\ell,a_{n-4},
  \ldots, a_1,i,j\}$ where the sign is $+$ if and only if
  $n \equiv 2,3 \pmod{4}$. %and $-1$ if $n \equiv 0,1 \pmod{4}$.
  We are done by (i).
\end{proof}
 
\begin{lemma}
  \label{lem:triple}
  If $1 \leqs i,j,k \leqs n-1$ distinct then
  \[ \{ i,j,k,a_1, \ldots, a_{n-3} \} + \{ j,k,i,a_1, \ldots, a_{n-3}
    \} + \{ k,i,j,a_1, \ldots, a_{n-3} \} = 0. \]
\end{lemma}
\begin{proof}
  This is proved by differentiating $\phi_1 \phi_2 =0$.
\end{proof}
 
We also have the analogues of Lemmas~\ref{lem:sym}(ii)
and~\ref{lem:triple} where each symbol is reversed. We write $1_A$ for
the indicator function of the event $A$.
 
We prove~\eqref{braceidentity1} by taking
$a_1, \ldots, a_{n-2} = 1, \ldots, \widehat{k}, \ldots, n-1$ in the
following lemma.
\begin{lemma}
 \label{lem:braces1}
 Let $1 \leqs i,j,k \leqs n-1$ distinct. Let $a_1, \ldots, a_{n-2},k$ be a
 permutation of $1,2, \ldots, n-1$. Then
 \begin{equation} \label{partial1} \frac{\partial^2 [a_1, \ldots,
     a_{n-2}]}{\partial x_i \partial x_j} = \pm \left(2 + 1_{i \in
       \{a_1,a_{n-2}\}} + 1_{j \in \{a_1,a_{n-2}\}} \right) \{ i,i,1,
   \ldots, \widehat{ijk}, \ldots, n-1,j,j\} \end{equation}
 and \begin{equation} \label{partial2} \frac{\partial^2 [[a_1, \ldots,
     a_{n-2}]]}{\partial x_i \partial x_j} = \pm 2n \{ i,i,1, \ldots,
   \widehat{ijk}, \ldots,n-1,j,j\}, \end{equation} where $\pm$ is the
 sign of the permutation taking $a_1, \ldots, a_{n-2}$ to
 $i,1, \ldots, \widehat{ijk}, \ldots, n-1,j$.
\end{lemma}

\begin{proof}
  We first prove~\eqref{partial1} when
  $\{a_1,a_{n-2}\} \cap \{i,j\} = \emptyset$. Using
  Lemma~\ref{lem:sym} we compute
 \begin{align*}
   \frac{\partial^2 [a_1, \ldots, a_{n-2}]}{\partial x_i \partial x_j} 
   &= \{ i, a_1, \ldots, a_{n-2},j \} + \{ j,a_1, \ldots, a_{n-2},i \}   \\ 
   &= \pm \left(\{ i, i, 1, \ldots,  \widehat{ijk}, \ldots, n-1,j,j \}
     - \{ j,j, 1, \ldots, \widehat{ijk}, \ldots, n-1,i,i \} \right) \\
   &= \pm 2 \{ i,i,1,  \ldots, \widehat{ijk}, \ldots,n-1,j,j\}.  
 \end{align*}
 If $a_1 = i$ and $a_{n-2} \not= j$ then the first term picks up a
 factor of $2$, and the second term is unchanged. If $a_1 = i$ and
 $a_{n-2} = j$ then the first term picks up a factor of $4$ and the
 second term vanishes by Lemma~\ref{lem:sym}(iii). The other cases are
 similar.
 
 We deduce~\eqref{partial2} from~\eqref{partial1} by summing over the
 $n-2$ terms in Definition~\ref{def:brackets}(ii).  Since $\sigma^2$
 is an even permutation, all the terms have the same sign.  There are
 two terms starting or ending in $i$, and two terms starting or ending
 in $j$. This gives an overall numerical factor of $2(n-2) + 2 + 2 = 2n$.
 \end{proof}
 
 We prove~\eqref{braceidentity2} by taking
 $a_1, \ldots, a_{n-2} = 1, \ldots, \widehat{j}, \ldots, n-1$ in the
 following lemma.
 \begin{lemma}
 \label{lem:braces2}
 Let $1 \leqs i,j \leqs n-1$ distinct. Let $a_1, \ldots, a_{n-2},j$ be a
 permutation of $1,2, \ldots, n-1$. Then
 \begin{equation*} %\label{partial3}
   \frac{\partial^2 [a_1, \ldots,
     a_{n-2}]}{\partial x_i^2} = \pm 2\left(1 + 1_{i \in
       \{a_1,a_{n-2}\}} \right) \{ i,i,1, \ldots, \widehat{ij},
   \ldots,n-1,i\} \end{equation*} and \begin{equation*} %\label{partial4}
   \frac{\partial^2 [[a_1, \ldots, a_{n-2}]]}{\partial x_i^2} = \pm 2n
   \{ i,i,1, \ldots, \widehat{ij}, \ldots,n-1,i\}, \end{equation*}
 where $\pm$ is the sign of the permutation taking
 $a_1, \ldots, a_{n-2}$ to $i,1, \ldots, \widehat{ij}, \ldots, n-1$.
 \end{lemma}
  
 \begin{proof}
   If $i \notin \{a_1, a_{n-2}\}$ then
   \[
     \frac{\partial^2 [a_1, \ldots, a_{n-2}]}{\partial x_i^2} = 2 \{
     i, a_1, \ldots, a_{n-2},i \} = \pm \{ i, i, 1, \ldots,
     \widehat{ij}, \ldots, n-1,i \}.  \] If $i \in \{a_1, a_{n-2}\}$
   then we pick up an extra factor of $2$.
   % This proves~\eqref{partial3}, and we
   % deduce~\eqref{partial4} exactly as before.
   This proves the result for $[ \cdots ]$.  
   We deduce the result for $[[ \cdots ]]$ exactly as before.
 \end{proof} 
  
We prove~\eqref{braceidentity3} by taking $r=0$ and
$b_1, \ldots, b_s = 1, \ldots, \widehat{jk}, \ldots, n-1$ in the
following lemma, and also using Lemma~\ref{double bracket permute}.

\begin{lemma}
  \label{lem:symbols3}
  Let $1 \leqs i,j,k \leqs n-1$ distinct. Let
  $a_{1}, \ldots, a_{r},b_{1}, \ldots, b_{s},j,k$ be a permutation of
  $1,2, \ldots, n-1$. Then
  \begin{align*} & \frac{\partial^2[a_1, \ldots, a_{r},k,b_{1},
      \ldots,b_{s}]}{\partial x_i \partial x_j} +
    \frac{\partial^2[a_1, \ldots, a_{r},j,b_{1},
      \ldots,b_{s}]}{\partial x_i \partial x_k} \\ & \qquad \qquad
    \qquad = \pm (-1)^s \left( 2 + 1_{ i \in \{a_1,b_s \}} +
      1_{rs=0} \right) \{i,i,1, \ldots, \widehat{ijk}, \ldots,
    n-1,j,k\},
  \end{align*}
  and
  \begin{align*} & \frac{\partial^2[[a_1, \ldots, a_{r},k,b_{1},
      \ldots,b_{s}]]}{\partial x_i \partial x_j} +
    \frac{\partial^2[[a_1, \ldots, a_{r},j,b_{1},
      \ldots,b_{s}]]}{\partial x_i \partial x_k} \\ & \qquad \qquad
    \qquad \qquad = \pm (-1)^s (2n) \{i,i,1, \ldots, \widehat{ijk},
    \ldots, n-1,j,k\},
  \end{align*}
  where $\pm$ is the sign of the permutation taking
  $a_1, \ldots, a_r,b_1,\ldots,b_s$ to
  $i,1, \ldots, \widehat{ijk}, \ldots, n-1$.
\end{lemma}
  
\begin{proof}
  We first suppose that $r,s \geqs 1$. Using Lemmas~\ref{lem:sym}
  and~\ref{lem:triple}, we compute
  \begin{align*}
    \{i,a_1, & \ldots, a_r,k,b_1, \ldots, b_s,j \}
      + \{ i,a_1, \ldots,a_r,j,b_1, \ldots,b_s,k\} \\
    &= (-1)^{s-1} \left( \{ i,a_1, \ldots,a_r,b_1, \ldots,b_{s-1},k,b_s,j \}
      + \{i,a_1, \ldots,a_r,b_1, \ldots,b_{s-1},j,b_s,k \} \right)  \\
    &= (-1)^s \{i,a_1, \ldots, a_r,b_1, \ldots,b_s,j,k \} \\
    &= \pm (-1)^s \{i,i,1, \ldots, \widehat{ijk}, \ldots, n-1,j,k\},
 \end{align*}
 and
 \begin{align*}
   \{j,a_1, & \ldots, a_r,k,b_1, \ldots, b_s,i \}
     + \{ k,a_1, \ldots,a_r,j,b_1, \ldots,b_s,i\} \\
   &= (-1)^{r-1} \left( \{ j,a_1,k,a_2, \ldots,a_r,b_1, \ldots, b_s,i \}
     + \{k,a_1,j,a_2, \ldots,a_r,b_1, \ldots,b_s,i \} \right) \\
   & = (-1)^r \{j,k,a_1, \ldots, a_r,b_1, \ldots,b_s,i \} \\
   & = \pm (-1)^{s-1} \{j,k,1, \ldots, \widehat{ijk}, \ldots, n-1,i,i \} \\
   & = \pm (-1)^{s} \{i,i,1, \ldots, \widehat{ijk}, \ldots, n-1,j,k\}.
 \end{align*}
 If $i \notin \{a_1,b_s\}$ then we simply add these two expressions,
 giving a factor of $2$. If $i \in \{a_1,b_s\}$ then we take twice one
 expression plus the other, giving a factor of $3$.

 We next suppose $s=0$. The first calculation in the last paragraph is
 modified by deleting the second line, and introducing a factor of~$2$
 thereafter.  If $a_1 \not= i$ then this gives an overall factor
 of~$3$. If $a_1 = i$ then we must take twice the first expression,
 but the second expression vanishes by Lemma~\ref{lem:sym}(iii). This
 gives an overall factor of $4$.  The case $r=0$ is similar.

 We deduce the result for $[[ \cdots ]]$ from that for $[ \cdots ]$ as
 before.
\end{proof}

The following lemma prepares for the proof of~\eqref{braceidentity4}.
\begin{lemma}
  \label{lem:indepsum}
  Let $1 \leqs i,j \leqs n-1$ distinct. Let $a_1, \ldots, a_{n-3}$ be a
  permutation of $1,2, \ldots, \widehat{ij},$ $\ldots, n-1$ with sign
  $\nu$. Then
  \begin{equation}
    \label{Aij}
    A(i,j) := \nu \left( \{ i,j,a_1, \ldots,a_{n-3},i\}
      + \{i,i,a_1, \ldots,a_{n-3}, j \} \right) \end{equation}
  does not depend on the choice of $a_1, \ldots, a_{n-3}$.
\end{lemma}

\begin{proof}
  If $\cdots$ denotes the same in each case, then by
  Lemma~\ref{lem:triple} we have
  \begin{align*}
     \{ ij \cdots rs i \} + \{ ij \cdots sri \} + \{ ij \cdots irs \} &= 0,\\
     \{ ii \cdots rs j \} + \{ ii \cdots srj \} + \{ ii \cdots jrs \} &= 0.
 \end{align*}
 Lemma~\ref{lem:sym}(ii) shows that the final terms in these two sums
 differ by a sign. The right hand side of~\eqref{Aij} is therefore
 invariant under switching the last two $a$'s.
 The lemma now follows by Lemma~\ref{lem:sym}(i).
\end{proof}
 
\begin{lemma}
\label{lem:ab}
Let $1 \leqs i,j \leqs n-1$ distinct. Let
$a_{1}, \ldots, a_{r},b_{1}, \ldots, b_{s},i,j$ be a permutation of
$1,2, \ldots, n-1$.  Then
\[\frac{\partial^2[a_1, \ldots, a_{r},j,b_{1}, \ldots,b_{s}]}{\partial x_i^2}
  + 2 \frac{\partial^2[a_1, \ldots,  a_{r},i,b_{1},
  \ldots,b_{s}]}{\partial x_i \partial x_j}  = \pm (-1)^r 2
  \left( 1 + 1_{rs=0} \right) A(i,j) \]
and 
\[\frac{\partial^2[[a_1, \ldots, a_{r},j,b_{1}, \ldots,b_{s}]]}{\partial x_i^2}
  + 2 \frac{\partial^2[[a_1, \ldots, a_{r},i,b_{1}, \ldots,
    b_{s}]]}{\partial x_i \partial x_j} = \pm (-1)^r (2n) A(i,j), \]
where $\pm$ is the sign of the permutation taking
$a_1, \ldots, a_r,b_1,\ldots,b_s$ to
$1, \ldots, \widehat{ij}, \ldots, n-1$.
\end{lemma} 
\begin{proof}
  If $r=0$ then the left hand side equals
\begin{equation}
\label{eqn:r=0}
2\{i,j,b_1, \ldots,b_s,i \} + 4 \{i,i,b_1, \ldots,b_s,j \}
+ 2 \{ j,i,b_1, \ldots, b_s,i \} = \pm 4 A(i,j). 
\end{equation}

If $r,s \geqs 1$ then we instead obtain
\[ 2\{i,a_1, \ldots, a_r,j,b_1, \ldots,b_s,i \} + 2 \{i,a_1, \ldots,
  a_r,i,b_1, \ldots,b_s,j \} + 2 \{ j,a_1, \ldots, a_r,i,b_1, \ldots,
  b_s,i \}. \] Cancelling a factor $(-1)^r 2$ and applying
Lemma~\ref{lem:sym} term by term gives
\[ -\{i,a_1,j, \ldots, a_r,b_1, \ldots,b_s,i \} + \{i,i,a_1, \ldots,
  a_r,b_1, \ldots,b_s,j \} - \{ j,a_1,i, \ldots, a_r,b_1, \ldots,
  b_s,i \}.  \] Applying Lemma~\ref{lem:triple} to the first and
third terms shows that this equals $\pm A(i,j)$.

Finally, when $s=0$ %(and so $r = n-3$)
the left hand side is
\[ 2\{i,a_1, \ldots,a_r,j,i \} + 2 \{i,a_1, \ldots,a_r,i,j \} + 4 \{
  j,a_1, \ldots, a_r,i,i \}. \] By Lemma~\ref{lem:sym}(iii) this is
minus the expression we get by replacing $b_1, \ldots, b_s$ by
$a_2, \ldots, a_r,a_1$ in~\eqref{eqn:r=0}.  This gives the factor
$(-1)^r$.

We deduce the result for $[[ \cdots ]]$ from that for $[ \cdots ]$ as
before.
\end{proof}

Taking $r=0$ and
$b_1, \ldots, b_s = 1, \ldots, \widehat{ij}, \ldots, n-1$ in
Lemma~\ref{lem:ab}, and appealing to Lemma~\ref{double bracket
  permute}, gives
\[ \frac{\partial^2 \Omega_i}{\partial x_i^2}
  - 2 \frac{\partial^2 \Omega_j}{\partial x_i \partial x_j} 
  = -\varepsilon_{ij} (2n) A(i,j). \]
Taking $a_1, \ldots,a_{n-3} = 1, \ldots, \widehat{ijk}, \ldots, n-1,k$
in Lemma~\ref{lem:indepsum} shows that this equals
\[ (-1)^{n+1} \varepsilon_{ijk} (2n) \left( \{i,j,1, \ldots,
    \widehat{ijk}, \ldots, n-1,k,i\} + \{i,i,1, \ldots, \widehat{ijk},
    \ldots, n-1,j,k \} \right). \]
Adding~\eqref{braceidentity3} gives~\eqref{braceidentity4}. This
completes the proof of Theorem~\ref{finaltheorem}, and hence of
Theorem~\ref{thm:overZ}.

\bibliographystyle{amsalpha}
\bibliography{references}
\end{document}